\documentclass[twoside,12pt]{article}

\usepackage{times,amsmath,amssymb,amsthm,cite}

\newcommand{\subj}[1]{\par\noindent{\bf AMS Subject Classifications: }#1.}
\newcommand{\keyw}[1]{\par\noindent{\bf Keywords: }#1.}

\numberwithin{equation}{section}
\numberwithin{figure}{section}

\newtheorem{theorem}{Theorem}[section]

\newtheorem{corollary}[theorem]{Corollary}

\theoremstyle{definition}
\newtheorem{definition}[theorem]{Definition}
\newtheorem{example}[theorem]{Example}
\theoremstyle{remark}
\newtheorem{remark}[theorem]{Remark}

\everymath{\displaystyle}
\date{}
\newcommand{\ijde}
{\vspace{-1in}\normalsize\flushleft
This is a preprint of a paper whose final and definite form will appear in\\
International Journal of Difference Equations, ISSN 0973-6069\\
{\tt http://campus.mst.edu/ijde}\\\vspace{1mm}\hrule\vspace{5mm}
\renewcommand\thefootnote{{}}

\footnotetext{\noindent\tt Received Nov 27, 2013; Revised Jan 24, 2014; Accepted Jan 31, 2014\par
\hspace*{8pt}Communicated by Ewa Schmeidel}}

\begin{document}

\title{\ijde\center\Large\bf
An Inverse Problem of the Calculus of Variations\\ on Arbitrary Time Scales}

\author{{\bf Monika Dryl}\\
Center for Research and Development in Mathematics and Applications\\
Department of Mathematics, University of Aveiro\\
3810--193 Aveiro, Portugal\\
{\tt monikadryl@ua.pt}\\[0.3cm]
{\bf Agnieszka B. Malinowska}\\
Faculty of Computer Science\\
Bialystok University of Technology\\
15--351 Bia\l ystok, Poland\\
{\tt a.malinowska@pb.edu.pl}\\[0.3cm]
{\bf Delfim F. M. Torres}\\
Center for Research and Development in Mathematics and Applications\\
Department of Mathematics, University of Aveiro\\
3810--193 Aveiro, Portugal\\
{\tt delfim@ua.pt}}

\maketitle

% -----------------------------------------------------------------------

\thispagestyle{empty}

\begin{abstract}
We consider an inverse extremal problem for variational functionals
on arbitrary time scales. Using the Euler--Lagrange equation and the strengthened
Legendre condition, we derive a general form for a variational functional
that attains a local minimum at a given point of the vector space.
\end{abstract}

\subj{34N05, 49N45}

\keyw{calculus of variations, Hilger's time-scale calculus, inverse problems}

\bibliographystyle{plain}

% -----------------------------------------------------------------------

\section{Introduction}

We study an inverse problem associated
with the following fundamental problem
of the calculus of variations: to minimize
\begin{equation}
\label{funct 1}
\mathcal{L}(y)=\int\limits_{a}^{b}
L\left(t,y^{\sigma}(t),y^{\Delta}(t)\right)\Delta t
\end{equation}
subject to the boundary conditions $y(a)=y_{0}(a)$, $y(b)=y_{0}(b)$,
on a given time scale $\mathbb{T}$. More precisely, we describe a general form
of a variational functional \eqref{funct 1} having an extremum at a given function $y_0$
under the Euler--Lagrange and strengthened Legendre conditions
on time scales \cite{BohnerCOVOTS}. Throughout the paper
we assume the reader to be familiar with the basic definitions and results
from the time scale theory \cite{BohnerDEOTS,MBbook2001,Hilger97}.
For a review on general approaches to the calculus of variations on time scales
see \cite{BohnerCOVOTS,MyID:252,china-Xuzhou,GirejkoMalinowska,Girejko,Malinowska,Martins,MyID:212}.
For analogous results in $\mathbb{T} = \mathbb{R}$ see \cite{orlov,orlov2}.
The results here obtained are new even for simple (but important)
time scales like $\mathbb{T} = \mathbb{Z}$ or $\mathbb{T} = q^{\mathbb{N}_{0}}$, $q > 1$.

The paper is organized as follows. In Section~\ref{sec:2} we collect some necessary
definitions and results of the delta calculus on time scales, which are used throughout
the text. The main results are presented in Section~\ref{sec:3}. We find a general
form of the variational functional \eqref{funct 1} that solves the inverse extremal problem
(Theorem~\ref{theorem general}). In order to illustrate our results, we present
the form of the Lagrangian $L$ on an isolated time scale (Corollary~\ref{cor1}).
We end by presenting the form of the Lagrangian $L$ in the periodic time scale
$\mathbb{T}=h \mathbb{Z}$, $h > 0$ (Example~\ref{cor hZ})
and in the $q$-scale $\mathbb{T}=q^{\mathbb{N}_{0}}$, $q>1$ (Example~\ref{ex1}).

% -----------------------------------------------------------------------

\section{Preliminaries}
\label{sec:2}

In this section we introduce basic definitions and theorems that will be useful in the sequel.
 A time scale $\mathbb{T}$ is an arbitrary nonempty closed subset of $\mathbb{R}$.
Let $a,b\in\mathbb{T}$ with $a<b$. We define the interval $[a,b]$ in $\mathbb{T}$ by
$[a,b]_{\mathbb{T}}:=[a,b]\cap\mathbb{T}=\left\{t\in\mathbb{T}: a\leq t\leq b\right\}$.

\begin{definition}[See \cite{BohnerDEOTS}]
\label{def:jump:op}
The forward jump operator $\sigma:\mathbb{T} \rightarrow \mathbb{T}$ is defined by
$\sigma(t):=\inf\lbrace s\in\mathbb{T}: s>t\rbrace$ for $t\neq \sup\mathbb{T}$
and $\sigma(\sup\mathbb{T}) := \sup\mathbb{T}$ if $\sup\mathbb{T}< +\infty$.
The backward jump operator $\rho:\mathbb{T} \rightarrow \mathbb{T}$ is
given by $\rho(t):=\sup\{s \in \mathbb{T}:s<t\}$ for $t\neq\inf \mathbb{T}$
and $\rho(\inf\mathbb{T})=\inf\mathbb{T}$ if $\inf\mathbb{T}>-\infty$.
The graininess function $\mu:\mathbb{T} \rightarrow [0,\infty)$
is defined by $\mu(t):=\sigma(t)-t$.
\end{definition}

A point $t\in\mathbb{T}$ is called \emph{right-dense},
\emph{right-scattered}, \emph{left-dense} or \emph{left-scattered}
if $\sigma(t)=t$, $\sigma(t)>t$, $\rho(t)=t$,
$\rho(t)<t$, respectively. We say that $t$ is \emph{isolated}
if $\rho(t)<t<\sigma(t)$, that $t$ is \emph{dense} if $\rho(t)=t=\sigma(t)$.

\begin{example}
The two classical time scales are $\mathbb{R}$ and $\mathbb{Z}$,
representing the continuous and the purely discrete time, respectively.
The other standard examples are $h\mathbb{Z}$, $h>0$, and $q^{\mathbb{N}_{0}}$, $q > 1$.
It follows from Definition~\ref{def:jump:op} that
if $\mathbb{T}=\mathbb{R}$, then
$\sigma (t)=t$ and $\mu(t) = 0$ for all $t \in \mathbb{T}$;
if $\mathbb{T}=h\mathbb{Z}$, then $\sigma(t)= t+h$
and $\mu(t) = h$ for all $t\in\mathbb{T}$; if $\mathbb{T}=q^{\mathbb{N}_{0}}$,
then $\sigma(t)=qt$ and $\mu(t)=t(q-1)$ for all $t\in\mathbb{T}$.
\end{example}

\begin{definition}[See \cite{MR2028477}]
A time scale $\mathbb{T}$ is said to be an isolated time scale provided given any
$t \in \mathbb{T}$, there is a $\delta > 0$ such that
$(t - \delta, t+\delta) \cap \mathbb{T} = \{t\}$.
\end{definition}

\begin{remark}
If the graininess function is bounded from below
by a strictly positive number, then the time scale is isolated \cite{MR2679122}.
Therefore, $h\mathbb{Z}$, $h > 0$, and $q^{\mathbb{N}_{0}}$, $q > 1$,
are examples of isolated time scales.
Note that the converse is not true.
For example, $\mathbb{T} = \log(\mathbb{N})$
is an isolated time scale but its graininess function
is not bounded from below by a strictly positive number.
\end{remark}

To simplify the notation, one usually uses $f^{\sigma}(t):=f(\sigma(t))$.
The delta derivative is defined for points from the set
$$
\mathbb{T}^{\kappa} :=
\begin{cases}
\mathbb{T}\setminus\left\{\sup\mathbb{T}\right\}
& \text{ if } \rho(\sup\mathbb{T})<\sup\mathbb{T}<\infty,\\
\mathbb{T}
& \hbox{ otherwise}.
\end{cases}
$$

\begin{definition}[See \cite{BohnerDEOTS}]
A function $f:\mathbb{T}\rightarrow\mathbb{R}$ is
\emph{$\Delta$-differentiable} at
$t\in\mathbb{T}^\kappa$ if there is a number $f^{\Delta}(t)$
such that for all $\varepsilon>0$ there exists a neighborhood $O$
of $t$ such that
$$
|f^\sigma(t)-f(s)-f^{\Delta}(t)(\sigma(t)-s)|
\leq\varepsilon|\sigma(t)-s|
\quad \mbox{ for all $s\in O$}.
$$
We call $f^{\Delta}(t)$ the \emph{$\Delta$-derivative} of $f$ at $t$.
\end{definition}

\begin{example}
If $\mathbb{T}=h\mathbb{Z}$, then
$f:\mathbb{T} \rightarrow \mathbb{R}$ is delta differentiable
at $t\in \mathbb{T}$ if, and only if,
\begin{equation*}
f^{\Delta}(t)=\frac{f(\sigma(t))-f(t)}{\mu(t)}=\frac{f(t+h)-f(t)}{h}=:\Delta_{h} f(t).
\end{equation*}
In the particular case $h=1$, $f^{\Delta}(t)=\Delta f(t)$,
where $\Delta$ is the usual forward difference operator.
If $\mathbb{T}=q^{\mathbb{N}_{0}}=\lbrace q^{k}:q>1, k\in\mathbb{N}_{0}\rbrace$,
then $f^{\Delta}(t)=\frac{f(qt)-f(t)}{(q-1)t}=:\Delta_{q}f(t)$, i.e.,
we get the usual Jackson derivative of quantum calculus \cite{QC}.
\end{example}

\begin{theorem}[See \cite{BohnerDEOTS}]
\label{differentiation}
Let $f:\mathbb{T} \rightarrow \mathbb{R}$
and $t\in\mathbb{T}^{\kappa}$.
If $f$ is delta differentiable at $t$, then
$$
f^\sigma(t)=f(t)+\mu(t)f^{\Delta}(t).
$$
\end{theorem}

\begin{definition}[See \cite{BohnerDEOTS}]
A function $f:\mathbb{T} \rightarrow \mathbb{R}$ is called rd-continuous provided
it is continuous at right-dense points in $\mathbb{T}$ and its left-sided limits exists
(finite) at all left-dense points in $\mathbb{T}$.
\end{definition}

The set of all rd-continuous functions $f:\mathbb{T} \rightarrow \mathbb{R}$
is denoted by $C_{rd} = C_{rd}(\mathbb{T}) = C_{rd}(\mathbb{T},\mathbb{R})$.
The set of functions $f:\mathbb{T} \rightarrow \mathbb{R}$ that are
$\Delta$-differentiable and whose derivative is rd-continuous is denoted by
$C^{1}_{rd}=C_{rd}^{1}(\mathbb{T})=C^{1}_{rd}(\mathbb{T},\mathbb{R})$.

A function $F:\mathbb{T}\rightarrow\mathbb{R}$ is called an antiderivative
of $f:\mathbb{T}\rightarrow\mathbb{R}$ provided that
$F^{\Delta}(t)=f(t)$ for all $t\in\mathbb{T}^{\kappa}$.
Let $\mathbb{T}$ be a time scale and $a,b\in\mathbb{T}$.
If $F$ is an antiderivative of $f$, then the Cauchy
$\Delta$-integral is defined by
$$
\int\limits_{a}^{b}f(t)\Delta t:=F(b)-F(a).
$$

\begin{theorem}[See \cite{BohnerDEOTS}]
Every rd-continuous function has an antiderivative.
In particular, if $t_{0}\in\mathbb{T}$, then $F$ defined by
\begin{equation*}
F(t):=\int\limits_{t_{0}}^{t} f(\tau)\Delta \tau,
\end{equation*}
$t\in\mathbb{T}$, is an antiderivative of $f$.
\end{theorem}

\begin{example}
If $\mathbb{T}=h\mathbb{Z}$, $h>0$,
and $a, b \in \mathbb{T}$ with $a<b$, then
\begin{equation*}
\int\limits_{a}^{b}f(t)\Delta t
=\sum\limits_{k=\frac{a}{h}}^{\frac{b}{h}-1}f(kh)h.
\end{equation*}
If $\mathbb{T}=q^{\mathbb{N}_{0}}$, $q>1$, then
$\displaystyle \int\limits_{a}^{b}f(t)\Delta t=(q-1)\sum_{t\in[a,b)\cap\mathbb{T}}tf(t)$.
\end{example}

Let $\mathbb{T}$ be a given time scale with at least three points.
Consider the following variational problem on the time scale $\mathbb{T}$:
\begin{equation}
\label{variational problem}
\mathcal{L}(y)=\int\limits_{a}^{b}
L\left(t, y^{\sigma}(t),y^{\Delta}(t)\right)\Delta t \longrightarrow \min,
\quad y(a)=\alpha, \quad y(b)=\beta,
\end{equation}
where $a,b\in\mathbb{T}$ with $a<b$; $\alpha, \beta\in\mathbb{R}^{n}$
with $n\in\mathbb{N}$, and $L:\mathbb{T}\times\mathbb{R}^{2n}\rightarrow\mathbb{R}$.

\begin{definition}
We say that $y\in C_{rd}^{1}(\mathbb{T})$ is admissible for problem
\eqref{variational problem} if it satisfies
the boundary conditions $y(a)=\alpha$ and $y(b)=\beta$.
\end{definition}

\begin{definition}
An admissible function $\hat{y}$ is called a \emph{local minimizer} of
problem \eqref{variational problem} provided there exists $\delta >0$ such that
$\mathcal{L}(\hat{y})\leq \mathcal{L}(y)$ for all admissible $y$ with
$\|y-\hat{y}\|_{C^{1}_{rd}}<\delta$, where
$$
\|f\|_{C^{1}_{rd}}
= \sup_{t\in[a,b]^{k}_{\mathbb{T}}}\|f^{\sigma}(t)\|
+\sup_{t\in[a,b]^{k}_{\mathbb{T}}}\|f^{\triangle}(t)\|
$$
with $\|\cdot\|$ a norm in $\mathbb{R}^n$.
\end{definition}

In what follows the Lagrangian $L$ is understood as a function
$(t,x,v) \rightarrow L(t,x,v)$ and by $L_x$ and $L_v$ we denote
the partial derivatives of $L$ with respect to $x$ and $v$, respectively.
Similar notation is used for second order partial derivatives.

\begin{theorem}[The Euler--Lagrange equation \cite{HilscgerZeidan}]
\label{E-L theorem}
Assume that $L(t,\cdot,\cdot)$ is differentiable in $(x,v)$
and $L(t,\cdot,\cdot)$, $L_{x}(t,\cdot,\cdot)$, $L_{v}(t,\cdot,\cdot)$
are continuous at $(y^{\sigma},y^{\Delta})$, uniformly in $t$
and rd-continuous in $t$ for any admissible $y$. If $\hat{y}(t)$
is a local minimizer of the variational problem \eqref{variational problem},
then there exists a vector $c\in\mathbb{R}^{n}$ such that the Euler--Lagrange equation
\begin{equation}
\label{E-L equation}
L_{v}\left(t, \hat{y}^{\sigma}(t),\hat{y}^{\Delta}(t)\right)
=\int\limits_{a}^{t}L_{x}(\tau, \hat{y}^{\sigma}(\tau),\hat{y}^{\Delta}(\tau))\Delta\tau
+c^{T}
\end{equation}
holds for $t\in[a,b]^{\kappa}_{\mathbb{T}}$.
\end{theorem}

\begin{theorem}[The Legendre condition \cite{BohnerCOVOTS}]
\label{Legendre theorem}
If $\hat{y}$ is a local minimizer of the variational problem
\eqref{variational problem}, then
\begin{equation}
\label{Legendre cond}
A(t)+\mu(t)\left\lbrace C(t)+C^{T}(t)+\mu(t)B(t)
+(\mu(\sigma(t)))^{\dag}A(\sigma(t))\right\rbrace\geq 0,
\end{equation}
$t\in[a,b]^{\kappa^{2}}_{\mathbb{T}}$, where
$A(t)=L_{vv}\left(t,\hat{y}^{\sigma}(t),\hat{y}^{\Delta}(t)\right)$,
$B(t)=L_{xx}\left(t,\hat{y}^{\sigma}(t),\hat{y}^{\Delta}(t)\right)$,
$C(t)=L_{xv}\left(t,\hat{y}^{\sigma}(t),\hat{y}^{\Delta}(t)\right)$,
and where $\alpha^{\dag}=\frac{1}{\alpha}$
if $\alpha\in\mathbb{R}\setminus\lbrace 0 \rbrace$ and $0^{\dag}=0$.
\end{theorem}

\begin{remark}
If \eqref{Legendre cond} holds with the strict inequality $>$,
then it is called \emph{the strengthened Legendre condition}.
\end{remark}

\begin{definition}[See \cite{BohnerDEOTS}]
We say that a function $p:\mathbb{T}\rightarrow\mathbb{R}$
is regressive provided
$$
1+\mu(t)p(t)\neq 0
$$
holds for all $t\in\mathbb{T}^{\kappa}$. The set of all regressive
and rd-continuous functions $f:\mathbb{T}\rightarrow\mathbb{R}$
is denoted by $\mathcal{R}=\mathcal{R}(\mathbb{T})
=\mathcal{R}(\mathbb{T},\mathbb{R})$.
\end{definition}

\begin{theorem}[See \cite{MBbook2001}]
\label{thm:IVP}
Let $p\in\mathcal{R}$, $f\in C_{rd}$, $t_{0}\in\mathbb{T}$
and $y_{0}\in\mathbb{R}$. Then, the unique solution
of the initial value problem
\begin{equation}
\label{eq:IVP}
y^{\Delta}=p(t)y+f(t),\quad y(t_{0})=y_{0},
\end{equation}
is given by
\begin{equation*}
y(t)=e_{p}(t,t_{0})y_{0}+\int\limits_{t_{0}}^{t}
e_{p}(t,\sigma(\tau))f(\tau)\Delta\tau,
\end{equation*}
where $e_{p}(\cdot,\cdot)$ denotes the exponential function
on time scales.
\end{theorem}

\begin{remark}[See \cite{BohnerDEOTS}]
\label{rem3}
An alternative form of the solution of the initial value problem \eqref{eq:IVP} is given by
\begin{equation*}
y(t)=e_{p}(t,t_{0})\left[y_{0}+\int\limits_{t_{0}}^{t}e_{p}(t_{0},\sigma(\tau))f(\tau)\Delta\tau\right].
\end{equation*}
\end{remark}

For more properties of the delta exponential function we refer the reader
to \cite{BohnerDEOTS,MBbook2001}.

% --------------------------------------------------------

\section{Main Results}
\label{sec:3}

The problem under our consideration
is to find a general form
of the variational functional
\begin{equation}
\label{functional}
\mathcal{L}(y)=\int\limits_{a}^{b}
L\left(t,y^{\sigma}(t),y^{\Delta}(t)\right)\Delta t,
\end{equation}
$L:[a,b]_{\mathbb{T}}\times\mathbb{R}^{2}\rightarrow\mathbb{R}$,
subject to the boundary conditions $y(a)=y(b)=0$,
possessing a local minimum at zero, under the Euler--Lagrange
and the strengthened Legendre conditions.
We assume that $L(t,\cdot,\cdot)$ is a $C^{2}$-function with respect to $(x,v)$
uniformly in $t$, and $L$, $L_{x}$, $L_{v}$, $L_{vv}\in C_{rd}$
for any admissible path $y(\cdot)$.
Observe that under our assumptions, by Taylor's theorem, we may
write $L$, with the big $O$ notation, in the form
\begin{equation}
\label{pre:integrand}
L(t, x,v)=P(t, x) +Q(t, x) v +\frac{1}{2} R(t, x,0)v^{2} + O(v^3),
\end{equation}
where
\begin{equation}
\label{notation}
\begin{gathered}
P(t, x) = L(t, x,0),\\
Q(t, x) = L_{v}(t, x,0),\\
R(t, x,0) = L_{vv}(t, x,0).
\end{gathered}
\end{equation}
Let $R(t, x, v) = R(t, x,0) + O(v)$.
Then, one can write \eqref{pre:integrand} as
\begin{equation}
\label{integrand}
L(t, x,v)=P(t, x) +Q(t, x) v +\frac{1}{2} R(t, x, v) v^{2}.
\end{equation}
Now the idea is to find general forms of $P(t, y^{\sigma}(t))$,
$Q(t, y^{\sigma}(t))$ and $R(t, y^{\sigma}(t), y^{\Delta}(t))$ using the Euler--Lagrange
and the strengthened Legendre conditions. Note that the Euler--Lagrange equation \eqref{E-L equation}
at the null extremal, with notation
\eqref{notation}, is
\begin{equation}
\label{E-Leq}
Q(t,0)=\int\limits_{a}^{t}P_{x}(\tau,0)\Delta\tau + C,
\end{equation}
$t\in [a,b]^{\kappa}_{\mathbb{T}}$.
Therefore, choosing an arbitrary function $P(t,y^{\sigma}(t))$ such that
$P(t,\cdot)\in C^{2}$ with respect to the second variable, uniformly in $t$,
$P$ and $P_x$ are rd-continuous in $t$ for all admissible $y$,
and by \eqref{E-Leq} we can write a general form of $Q$:
\begin{equation}
\label{fullQ}
Q(t,y^{\sigma}(t))=C+\int\limits_{a}^{t}P_{x}(\tau,0)\Delta \tau
+q(t,y^{\sigma}(t))-q(t,0),
\end{equation}
where $C\in\mathbb{R}$ and $q$ is an arbitrary function such that $q(t,\cdot)\in C^{2}$
with respect to the second variable, uniformly in $t$,
and $q$ and $q_x$ are rd-continuous in $t$ for all admissible $y$.
With notation \eqref{notation}, the strengthened
Legendre condition \eqref{Legendre cond} at the null extremal has the form
\begin{equation}
\label{Legendre}
R(t,0,0)+\mu(t)\left\lbrace 2Q_{x}(t,0)+\mu(t)P_{xx}(t,0)
+\left(\mu^{\sigma}(t)\right)^{\dag}R(\sigma(t),0,0)\right\rbrace > 0,
\end{equation}
$t\in [a,b]^{\kappa^{2}}_{\mathbb{T}}$, where $\alpha^{\dag}=\frac{1}{\alpha}$
if $\alpha\in\mathbb{R}\setminus\lbrace 0 \rbrace$ and $0^{\dag}=0$.
Hence, we set
\begin{equation}
\label{eq3}
R(t,0,0)+\mu(t)\left\lbrace 2Q_{x}(t,0)+\mu(t)P_{xx}(t,0)
+\left(\mu^{\sigma}(t)\right)^{\dag}R(\sigma(t),0,0)\right\rbrace = p(t)
\end{equation}
with $p\in C_{rd}^{1}([a,b]_{\mathbb{T}})$, $p(t)>0$ for all $t\in [a,b]^{\kappa^{2}}_{\mathbb{T}}$, chosen arbitrary.
Note that there exists a unique solution of \eqref{eq3} with
respect to $R(t, 0, 0)$.
If $t$ is a right-dense point, then $\mu(t)=0$ and $R(t,0,0) = p(t)$.
Otherwise, $\mu(t) \ne 0$, and using Theorem~\ref{differentiation}
with $f(t) = R(t,0,0)$ we modify equation \eqref{eq3} into a first order
delta dynamic equation, which has a unique solution $R(t,0,0)$ in agreement
with Theorem~\ref{thm:IVP} (see details in the proof of Corollary~\ref{cor1}).
We derive a general form of $R$ from Legendre's condition \eqref{Legendre},
as a sum of the solution $R(t,0,0)$ of equation \eqref{eq3} and function $w$,
which is chosen arbitrarily in such a way that $w(t,\cdot,\cdot)\in C^2$
with respect to the second and the third variable,
uniformly in $t$; $w_x,w_v$ and $w_{vv}$ are rd-continuous
in $t$ for all admissible $y$.
Concluding: a general form of the integrand $L$ for functional \eqref{functional}
follows from \eqref{integrand}, \eqref{fullQ} and \eqref{eq3}, and is given by
\begin{equation}
\label{eq:genF:lag}
\begin{split}
L&\left(t,y^{\sigma}(t),y^{\Delta}(t)\right)
= P\left(t,y^{\sigma}(t)\right)\\
&+\left(C+\int\limits_{a}^{t}P_{x}(\tau,0)\Delta \tau
+q(t,y^{\sigma}(t))-q(t,0)\right)y^{\Delta}(t)\\
&+\Biggl(p(t)-\mu(t)\left\lbrace 2Q_{x}(t,0)+\mu(t) P_{xx}(t,0)
+\left(\mu^{\sigma}(t)\right)^{\dag}R(\sigma(t),0,0)\right\rbrace\\
&+w(t,y^{\sigma}(t),y^{\Delta}(t))
-w(t,0,0)\Biggr)\frac{y^{\Delta}(t)^{2}}{2}.
\end{split}
\end{equation}
We have just proved the following result.

\begin{theorem}
\label{theorem1}
Let $\mathbb{T}$ be an arbitrary time scale. If functional \eqref{functional}
with boundary conditions $y(a)=y(b)=0$ attains a local minimum at $\hat{y}(t)\equiv 0$
under the strengthened Legendre condition, then its Lagrangian $L$ takes the form \eqref{eq:genF:lag},
where $R(t,0,0)$ is a solution of equation \eqref{eq3}, $C\in\mathbb{R}$, $\alpha^{\dag}=\frac{1}{\alpha}$
if $\alpha\in\mathbb{R}\setminus\lbrace 0 \rbrace$ and $0^{\dag}=0$.
Functions $P$, $p$, $q$ and $w$ are arbitrary functions satisfying:
\begin{itemize}
\item[(i)]\
$P(t,\cdot),q(t,\cdot)\in C^{2}$ with respect to the second variable uniformly in $t$;
$P$, $P_x$, $q$, $q_x$ are rd-continuous in $t$ for all admissible $y$;
$P_{xx}(\cdot,0)$ is rd-continuous in $t$; $p\in C^{1}_{rd}$ with
$p(t)>0$ for all $t\in [a,b]^{\kappa^{2}}_{\mathbb{T}}$;
\item[(ii)]\
$w(t,\cdot,\cdot)\in C^2$ with respect to the second and the third variable, uniformly in $t$,
$w_x,w_v,w_{vv}$ are rd-continuous in $t$ for all admissible $y$.
\end{itemize}
\end{theorem}

Now we consider the general situation when the variational problem consists in minimizing
\eqref{functional} subject to arbitrary boundary conditions $y(a)=y_{0}(a)$ and $y(b)=y_{0}(b)$,
for a certain given function $y_{0}\in C_{rd}^{2}([a,b]_{\mathbb{T}})$.

\begin{theorem}
\label{theorem general}
Let $\mathbb{T}$ be an arbitrary time scale. If the variational functional
\eqref{functional} with boundary conditions $y(a)=y_{0}(a)$, $y(b)=y_{0}(b)$,
attains a local minimum for a certain given function
$y_{0}(\cdot)\in C^{2}_{rd}([a,b]_\mathbb{T})$
under the strengthened Legendre condition, then its Lagrangian $L$ has the form
\begin{equation*}
\begin{split}
&L\left(t,y^{\sigma}(t),y^{\Delta}(t)\right)
= P\left(t,y^{\sigma}(t)-y^{\sigma}_{0}(t)\right)
+ \left(y^{\Delta}(t)-y_{0}^{\Delta}(t)\right)\\
&\times \left(C+\int\limits_{a}^{t}P_{x}\left(\tau,-y_{0}^{\sigma}(\tau)\right)\Delta \tau
+q\left(t,y^{\sigma}(t)-y^{\sigma}_{0}(t)\right)
-q\left(t,-y_{0}^{\sigma}(t)\right)\right)\\
&+\frac{1}{2}\Biggl(p(t)-\mu(t)\left\lbrace 2Q_{x}(t,0)+\mu(t) P_{xx}(t,0)
+\left(\mu^{\sigma}(t)\right)^{\dag}R(\sigma(t),0,0)\right\rbrace
\\
&+w(t,y^{\sigma}(t)-y^{\sigma}_{0}(t),y^{\Delta}(t)-y_{0}^{\Delta}(t))
-w\left(t,-y_{0}^{\sigma}(t),-y_{0}^{\Delta}(t)\right)\Biggr)
\left(y^{\Delta}(t)-y_{0}^{\Delta}(t)\right)^{2},
\end{split}
\end{equation*}
where $R(t,0,0)$ is the solution of equation \eqref{eq3}, $C\in\mathbb{R}$
and functions $P$, $p$, $q$, $w$ satisfy conditions (i) and (ii) of Theorem~\ref{theorem1}.
\end{theorem}

\begin{proof}
The result follows as a corollary of Theorem~\ref{theorem1}.
In order to reduce the problem to the case of null boundary conditions
$y(a)=0$ and $y(b)=0$, we introduce the auxiliar variational functional
\begin{equation*}
\begin{split}
\tilde{\mathcal{L}}(y)&:= \mathcal{L}(y+y_0)=
\int\limits_{a}^{b}
L\left(t,y^{\sigma}(t)+y_{0}^{\sigma}(t),y^{\Delta}(t)+y_{0}^{\Delta}(t)\right)\Delta t\\
&=: \int\limits_{a}^{b} \tilde{L}\left(t,y^{\sigma}(t),y^{\Delta}(t)\right)\Delta t
\end{split}
\end{equation*}
subject to boundary conditions $y(a)=0$ and $y(b)=0$.
The result follows by application of Theorem~\ref{theorem1} to the auxiliar
Lagrangian $\tilde{L}$.
\end{proof}

For the classical situation $\mathbb{T}=\mathbb{R}$,
Theorem~\ref{theorem general} gives a recent result of \cite{orlov}.

\begin{corollary}[Theorem~4 of \cite{orlov}]
\label{cor R}
If the variational functional
$$
\mathcal{L}(y)=\int\limits_{a}^{b} L(t,y(t),y'(t))dt
$$
attains a local minimum at $y_{0}(\cdot)\in C^{2}[a,b]$
satisfying boundary conditions $y(a)=y_{0}(a)$ and $y(b)=y_{0}(b)$
and the classical Legendre condition $R(t,y_{0}(t),y'_{0}(t))>0$, $t\in[a,b]$,
then its Lagrangian $L$ has the form
\begin{multline*}
L(t,y(t),y^{'}(t))=P(t,y(t)-y_{0}(t))\\
+(y^{'}(t)-y_{0}^{'}(t))\left(C+\int\limits_{a}^{t}
P_{x}(\tau,-y_{0}(\tau))d\tau+q(t,y(t)-y_{0}(t))-q(t,-y_{0}(t))\right)\\
+\frac{1}{2}\left(p(t)+w(t,y(t)-y_{0}(t),y^{'}(t)-y_{0}^{'}(t))
-w(t,-y_{0}(t),-y_{0}^{'}(t))\right)(y^{'}(t)-y_{0}^{'}(t))^{2},
\end{multline*}
where $C\in\mathbb{R}$.
\end{corollary}

\begin{proof}
Follows from Theorem~\ref{theorem general} with $\mathbb{T}=\mathbb{R}$.
\end{proof}

Theorem~\ref{theorem general} seems to be new for any time scale other than $\mathbb{T}=\mathbb{R}$.
In the particular case of an isolated time scale, where $\mu(t) \neq 0$ for all $t\in\mathbb{T}$,
we get the following corollary.

\begin{corollary}
\label{cor1}
Let $\mathbb{T}$ be an isolated time scale. If functional \eqref{functional}
subject to the boundary conditions $y(a)=y(b)=0$ attains a local minimum at
$\hat{y}(t) \equiv 0$ under the strengthened Legendre condition,
then the Lagrangian $L$ has the form
\begin{equation}
\begin{split}
\label{eq4}
&L\left(t,y^{\sigma}(t),y^{\Delta}(t)\right)
= P\left(t,y^{\sigma}(t)\right)\\
&+\left(C+\int\limits_{a}^{t}P_{x}(\tau,0)\Delta \tau
+q(t,y^{\sigma}(t))-q(t,0)\right)y^{\Delta}(t)\\
&+\left(e_{r}(t,a)R_{0}
+\int\limits_{a}^{t}e_{r}(t,\sigma(\tau))s(\tau)\Delta\tau
+w(t,y^{\sigma}(t),y^{\Delta}(t))
-w(t,0,0)\right)\frac{y^{\Delta}(t)^{2}}{2},
\end{split}
\end{equation}
where $C,R_{0}\in\mathbb{R}$ and $r(t)$ and $s(t)$ are given by
\begin{equation}
\label{funct s, r}
r(t) := -\frac{1+\mu(t)(\mu^{\sigma}(t))^{\dag}}{\mu^{2}(t)(\mu^{\sigma}(t))^{\dag}},
\quad s(t) := \frac{p(t)
-\mu(t)[2Q_{x}(t,0)+\mu(t)P_{xx}(t,0)]}{\mu^{2}(t)(\mu^{\sigma}(t))^{\dag}},
\end{equation}
with $\alpha\in\mathbb{R}\setminus\lbrace 0 \rbrace$ and $0^{\dag}=0$,
where functions $P$, $p$, $q$, $w$ satisfy assumptions of Theorem~\ref{theorem1}.
\end{corollary}

\begin{proof}
In the case of an isolated time scale $\mathbb{T}$,
we may obtain the form of function $Q$ in the same
way as it was done in the proof of Theorem~\ref{theorem1}.
We derive a general form for $R$ from Legendre's condition.
By relation $f^{\sigma}=f+\mu f^{\Delta}$
(Theorem~\ref{differentiation}), one may write equation \eqref{eq3} as
\begin{multline*}
R(t,0,0)+\mu(t)(\mu^{\sigma}(t))^{\dag}\left(R(t,0,0)+\mu(t) R^{\Delta}(t,0,0)\right)\\
+\mu(t) \left\lbrace 2Q_{x}(t,0)+\mu(t) P_{xx}(t,0)\right\rbrace - p(t)=0.
\end{multline*}
Hence,
\begin{multline}
\label{eq1}
\mu^{2}(t)(\mu^{\sigma}(t))^{\dag}R^{\Delta}(t,0,0)
+\left[1+\mu(t)(\mu^{\sigma}(t))^{\dag}\right]R(t,0,0)\\
+\mu(t)[2Q_{x}(t,0)+\mu(t) P_{xx}(t,0)]-p(t)=0.
\end{multline}
For an isolated time scale $\mathbb{T}$, equation \eqref{eq1}
is a first order delta dynamic equation of the following form:
$$
R^{\Delta}(t,0,0)+\frac{1+\mu(t)(\mu^{\sigma}(t))^{\dag}}{\mu^{2}(t)(\mu^{\sigma}(t))^{\dag}}R(t,0,0)
+\frac{\mu(t)[2Q_{x}(t,0)+\mu(t)P_{xx}(t,0)]-p(t)}{\mu^{2}(t)(\mu^{\sigma}(t))^{\dag}}=0.
$$
With notation \eqref{funct s, r} we have
\begin{equation}
\label{eq2}
R^{\Delta}(t,0,0)=r(t)R(t,0,0)+s(t).
\end{equation}
Observe that $r(t)$ is regressive. Indeed, if $\mu(t)\neq 0$, then
$$
1+\mu(t)r(t)=1-\frac{1+\mu(t)(\mu^{\sigma}(t))^{\dag}}{\mu(t)(\mu^{\sigma}(t))^{\dag}}
=1-\frac{\mu^{\sigma}(t)+\mu(t)}{\mu(t)}
=-\frac{\mu^{\sigma}(t)}{\mu(t)}\neq 0
$$
for all $t\in [a,b]^{\kappa}$.
Therefore, by Theorem~\ref{thm:IVP},
there is a unique solution to equation \eqref{eq2} with initial condition
$R(a,0,0)=R_{0}\in\mathbb{R}$:
\begin{equation}
\label{eq5}
R(t,0,0)=e_{r}(t,a)R_{0}+\int\limits_{a}^{t}e_{r}(t,\sigma(\tau))s(\tau)\Delta\tau .
\end{equation}
Thus, a general form of the integrand $L$ for functional \eqref{functional}
is given by \eqref{eq4}.
\end{proof}

\begin{remark}
\label{rem4}
Instead of \eqref{eq5}, we can use an alternative form
for the solution of the initial value problem
\eqref{eq2} subject to $R(a,0,0)=R_{0}$
(cf. Remark~\ref{rem3}):
\begin{equation*}
R(t,0,0)=e_{r}(t,a)\left[ R_{0}
+\int\limits_{a}^{t}e_{r}(a,\sigma(\tau))s(\tau)\Delta\tau\right].
\end{equation*}
Then the Lagrangian $L$ \eqref{eq4} can be written as
\begin{equation*}
\begin{split}
&L\left(t,y^{\sigma}(t),y^{\Delta}(t)\right)
= P\left(t,y^{\sigma}(t)\right)\\
&+\left(C+\int\limits_{a}^{t}P_{x}(\tau,0)\Delta \tau
+q(t,y^{\sigma}(t))-q(t,0)\right)y^{\Delta}(t)\\
&+\left(e_{r}(t,a)\left[ R_{0}
+\int\limits_{a}^{t}e_{r}(a,\sigma(\tau))s(\tau)\Delta\tau\right]
+w(t,y^{\sigma}(t),y^{\Delta}(t))
-w(t,0,0)\right)\frac{y^{\Delta}(t)^{2}}{2}.
\end{split}
\end{equation*}
\end{remark}

Based on Corollary~\ref{cor1}, we present the form of Lagrangian $L$
in the periodic time scale $\mathbb{T}=h\mathbb{Z}$.

\begin{example}
\label{cor hZ}
Let $\mathbb{T}=h\mathbb{Z}$, $h > 0$,
and $a, b\in h\mathbb{Z}$ with $a<b$.
Then $\mu(t) \equiv h$.
We consider the variational functional
\begin{equation}
\label{functional hZ}
\mathcal{L}(y)=h\sum_{k=\frac{a}{h}}^{\frac{b}{h}-1}
L\left(kh,y(kh+h),\Delta_h y(kh)\right)
\end{equation}
subject to the boundary conditions $y(a)=y(b)=0$, which
attains a local minimum at $\hat{y}(kh)\equiv 0$
under the strengthened Legendre condition
$$
R(kh,0,0)+2hQ_{x}(kh,0)+h^{2}P_{xx}(kh,0)+R(kh+h,0,0)>0,
$$
$kh\in [a,b-2h]\cap h\mathbb{Z}$.
Functions $r(t)$ and $s(t)$  (see \eqref{funct s, r})
have the following form:
\begin{equation*}
r(t)=\frac{-2}{h}\in\mathcal{R},
\quad s(t)=\frac{p(t)}{h}
-\left(2Q_{x}(t,0) + h P_{xx}(t,0)\right).
\end{equation*}
Hence,
$$
\int\limits_{a}^{t}P_{x}(\tau,0)\Delta \tau
=h\sum\limits_{i=\frac{a}{h}}^{\frac{t}{h}-1}P_{x}(ih,0),
$$
\begin{equation*}
\int\limits_{a}^{t}e_{r}(t,\sigma(\tau))s(\tau)\Delta \tau
=\sum_{i=\frac{a}{h}}^{\frac{t}{h}-1}(-1)^{\frac{t}{h}-i-1}
\left(p(ih)-2hQ_{x}(ih,0)-h^{2}P_{xx}(ih,0)\right).
\end{equation*}
Therefore, the Lagrangian $L$ of the variational functional \eqref{functional hZ}
on $\mathbb{T}=h\mathbb{Z}$ has the form
\begin{equation*}
\begin{split}
L&\left(kh,y(kh+h),\Delta_h y(kh)\right)=P\left(kh,y(kh+h)\right)\\
&+\left(C+\sum\limits_{i=\frac{a}{h}}^{k-1}hP_{x}(ih,0)
+q(kh,y(kh+h))-q(kh,0)\right)\Delta_h y(kh)\\
&+\frac{1}{2}\Biggl((-1)^{k-\frac{a}{h}}R_{0}
+\sum_{i=\frac{a}{h}}^{k-1}(-1)^{k-i-1}
\left(p(i h)-2hQ_{x}(ih,0)-h^{2}P_{xx}(ih,0)\right)\\
&+w(kh,y(kh+h),\Delta_h y(kh))-w(kh,0,0)\Biggr)
\left(\Delta_h y(kh)\right)^{2},
\end{split}
\end{equation*}
where functions $P$, $p$, $q$, $w$ are arbitrary but satisfy assumptions of Theorem~\ref{theorem1}.
\end{example}

Now we consider the $q$-scale $\mathbb{T}=q^{\mathbb{N}_{0}}$, $q > 1$.
In order to present the form of Lagrangian $L$, we use Remark~\ref{rem4}.

\begin{example}
\label{ex1}
Let $\mathbb{T}=q^{\mathbb{N}_{0}}=\lbrace q^{k}: q>1, k\in\mathbb{N}_{0}\rbrace$
and $a, b \in \mathbb{T}$ with $a < b$. We consider the variational functional
\begin{equation}
\label{eq6}
\mathcal{L}(y)=(q-1) \sum_{t \in [a,b)} t L\left(t,y(qt),\Delta_{q}y(t)\right)
\end{equation}
subject to the boundary conditions $y(a)=y(b)=0$, which attains a local minimum
at $\hat{y}(t) \equiv 0$ under the strengthened Legendre condition
\begin{equation*}
R(t,0,0)+(q-1)t\lbrace 2Q_{x}(t,0)+(q-1)t P_{xx}(t,0)\rbrace+\frac{1}{q}R(qt,0,0)>0
\end{equation*}
at the null extremal, $t\in \left[a,\frac{b}{q^{2}}\right] \cap q^{\mathbb{N}_{0}}$.
Functions given by \eqref{funct s, r} may be written as
$$
r(t)=\frac{q+1}{t(1-q)},\quad s(t)=\frac{qp(t)}{t(q-1)}-2qQ_{x}(t,0)-q(q-1)tP_{xx}(t,0).
$$
Hence,
$$
\int\limits_{a}^{t}P_{x}(\tau,0)\Delta\tau
=(q-1) \sum_{\tau \in [a,t)} \tau P_{x}(\tau,0),
\quad
e_{r}(t,a)=\prod_{s\in[a,t)} (-q),
$$
\begin{multline*}
\int\limits_{a}^{t}e_{r}(a,\sigma(\tau))s(\tau)\Delta \tau\\
=\sum_{\tau \in [a,t)} \frac{(1-q) \tau}{q \prod\limits_{s\in [a,\tau)} (-q)}
\left[\frac{q p(\tau)}{\tau(q-1)}-2q Q_{x}(\tau,0)-q(q-1)\tau P_{xx}(\tau,0)\right].
\end{multline*}
Therefore, the Lagrangian $L$ of the variational functional
\eqref{eq6} has the form
\begin{multline*}
L(t,y(qt),\Delta_{q}y(t))=P(t, y(qt))\\
+\left(
C+(q-1)\sum_{\tau \in [a,t)} \tau P_{x}(\tau,0)+ q(t,y(qt))-q(t,0)
\right)\Delta_{q}y(t) +\Bigg\lbrace
\prod_{s\in[a,t)}\left(-q\right)\\
\times \left[
R_{0}+ \sum_{\tau \in [a,t)}
\frac{(1-q)\tau}{q \prod\limits_{s\in [a,\tau)} (-q)}
\left(\frac{q p(\tau)}{\tau(q-1)}-2q Q_{x}(\tau,0)-q(q-1)\tau P_{xx}(\tau,0)
\right)\right]\\
+w\left(t,y(qt),\Delta_{q}y(t)\right)-w(t,0,0)
\Bigg\rbrace \frac{(\Delta_{q} y(t))^{2}}{2},
\end{multline*}
where functions $P$, $p$, $r$, $w$ are arbitrary but satisfy assumptions of Theorem~\ref{theorem1}.
\end{example}

% -----------------------------------------------

\section*{Acknowledgments}

This article was supported by Portuguese funds through the
\emph{Center for Research and Development in Mathematics and Applications} (CIDMA),
and \emph{The Portuguese Foundation for Science and Technology} (FCT),
within project PEst-OE/MAT/UI4106/2014. Dryl was also supported
by FCT through the Ph.D. fellowship SFRH/BD/51163/2010;
Malinowska by Bialystok University of Technology grant S/WI/02/2011;
and Torres by the FCT project PTDC/EEI-AUT/1450/2012,
co-financed by FEDER under POFC-QREN with
COMPETE ref. FCOMP-01-0124-FEDER-028894.
The authors are very grateful to two anonymous referees
for valuable remarks and comments, which
significantly contributed to the quality of the paper.

% -----------------------------------------------

%-----------------------------------------------

\end{document}